\documentclass[11pt,reqno]{article}
\usepackage{latexsym, amsmath, amssymb, a4, epsfig}
\usepackage{amsthm}
\usepackage{graphicx}
\usepackage{latexsym, color}

\newtheorem{thm}{Theorem}[section]
\newtheorem{cor}[thm]{Corollary}

\newtheorem{prop}[thm]{Proposition}
\newtheorem{rem}[thm]{Remark}

\usepackage[left=1in,right=1in,top=1in,bottom=1in]{geometry}

\newcommand{\SD}{\mathcal{S}_{\partial D}}
\newcommand{\SO}{\mathcal{S}_{\partial \Omega}}

\newcommand{\KDS}{\mathcal{K}_{\partial D}^{*}}

\newcommand{\KOS}{\mathcal{K}_{\partial \Omega}^{*}}

\newcommand{\bef}{\begin{figure}}
\newcommand{\enf}{\end{figure}}

\newcommand{\ds}{\displaystyle}

\newcommand{\p}{\partial}

\newcommand{\eqnref}[1]{(\ref {#1})}

\newcommand{\Rbb}{\mathbb{R}}

\newcommand{\la}{\langle}
\newcommand{\ra}{\rangle}

\newcommand{\Acal}{\mathcal{A}}

\newcommand{\Kcal}{\mathcal{K}}

\newcommand{\Scal}{\mathcal{S}}

\newcommand{\Ga}{\alpha}
\newcommand{\Gb}{\beta}
\newcommand{\Gd}{\delta}

\newcommand{\Gve}{\varepsilon}

\newcommand{\Gvf}{\varphi}
\newcommand{\Gg}{\gamma}

\newcommand{\Gl}{\lambda}

\newcommand{\Gm}{\mu}

\newcommand{\Gr}{\rho}

\newcommand{\Gs}{\sigma}

\newcommand{\Go}{\omega}

\newcommand{\GD}{\Delta}

\newcommand{\GO}{\Omega}

\newcommand{\beq}{\begin{equation}}
\newcommand{\eeq}{\end{equation}}

\def\ol{\overline}

\numberwithin{equation}{section}
\numberwithin{figure}{section}

\begin{document}
\title{Polarization tensor vanishing structure of general shape: Existence for small perturbations of balls\thanks{\footnotesize This paper is a revised version (with the title changed) of the earlier manuscript arXiv:1911.07250. This work is supported by NRF grants No. 2017R1A4A1014735 and 2019R1A2B5B01069967,  JSPS KAKENHI Grant No. 18H01126, NSF of China grant No. 11901523, and a grant from Central South University.}}

\author{Hyeonbae Kang\thanks{Department of Mathematics and Institute of Applied Mathematics, Inha University, Incheon 22212, S. Korea, and Department of Mathematics and Statistics, Central South University, Changsha, Hunan, P.R. China (hbkang@inha.ac.kr).}  \and Xiaofei Li\thanks{College of Science, Zhejiang University of Technology, Hangzhou, 310023, P. R. China (xiaofeilee@hotmail.com).}  \and Shigeru Sakaguchi\thanks{Research Center for Pure and Applied Mathematics, Graduate School of Information Sciences, Tohoku University, Sendai, 980-8579, Japan (sigersak@tohoku.ac.jp).}}

\date{\today}
\maketitle

\begin{abstract}
The polarization tensor is a geometric quantity associated with a domain. It is a signature of the small inclusion's existence inside a domain and used in the small volume expansion method to reconstruct small inclusions by boundary measurements. In this paper, we consider the question of the polarization tensor vanishing structure of general shape. The only known examples of the polarization tensor vanishing structure are concentric disks and balls. We prove, by the implicit function theorem on Banach spaces, that a small perturbation of a ball can be enclosed by a domain so that the resulting inclusion of the core-shell structure becomes polarization tensor vanishing. The boundary of the enclosing domain is given by a sphere perturbed by spherical harmonics of degree zero and two. This is a continuation of the earlier work \cite{KLS2D} for two dimensions.
\end{abstract}

\noindent{\footnotesize {\bf AMS subject classifications}. 35Q60 (primary); 31B10, 35B40, 35R30, 35R05 (secondary)}

\noindent{\footnotesize {\bf Key words}. Polarization tensor, polarization tensor vanishing structure, weakly neutral inclusion, neutral inclusion, existence, perturbation of balls, implicit function theorem, invisibility cloaking}

\section{Introduction}

In the inverse conductivity problem or the electrical impedance tomography, the measurement of boundary data is utilized to reconstruct inclusions buried inside the domain. When the inclusion is of small size, the small volume expansion shows that the leading order term of the boundary perturbation is expressed by the polarization tensor (abbreviated by PT afterwards) associated with the inclusion. Thus the polarization tensor is a signature of inclusion's existence, which can be effectively used to reconstruct the inclusion (see, for example, \cite{AK04, AK11, AKKL, BP, CV, CMV, VV}).

This paper is concerned with the problem of the opposite direction: hiding inclusions by making the PT vanish. Since the PT for simply connected homogeneous domain is either positive- or negative-definite, we consider the inclusions of core-shell structure. It is known that concentric disks and balls can be made to be PT vanishing (see \eqnref{neutral} below), and these are the only known examples of the PT-vanishing inclusions.

We are concerned with the following question:

\begin{itemize}
\item[]
{\bf Polarization Tensor Vanishing Structure}. {\sl Find a domain $\GO$ enclosing the given domain $D$ of arbitrary shape so that polarization tensor of the resulting inclusion $(D,\GO)$ of the core-shell structure vanishes.}
\end{itemize}

\noindent The purpose of this paper is to prove that if the core $D$ is a small perturbation of a ball in three dimensions, then there is $\GO$ enclosing $\ol{D}$ such that the inclusion $(D,\GO)$ becomes PT-vanishing. This is a continuation of the work \cite{KLS2D}, where a similar result is proved in two dimensions. So, we move directly to description of the problem and statement of the result leaving additional motivational remarks and historical accounts to that paper and references therein (see also recent survey article \cite{JKLS}).

To define the PT-vanishing structure of the core-shell shape, let $D$ and $\GO$ be bounded simply connected domains in $\Rbb^d$ ($d=2,3$) such that $\ol{D}\subset \GO$. The pair $(D, \GO)$ of domains may be regarded as an inclusion of the core-shell structure where the core $D$ is coated by the shell $\GO\setminus D$. Let $\Gs$ be a piecewise constant function, representing the conductivity distribution, defined by
\beq\label{conductivity}
\Gs=
\begin{cases}
\Gs_c \quad &\mbox{in } D, \\
\Gs_s \quad &\mbox{in } \GO\setminus D, \\
\Gs_m \quad &\mbox{in } \Rbb^d\setminus \GO,
\end{cases}
\eeq
where the conductivities $\Gs_c$, $\Gs_s$ and $\Gs_m$ of the core, the shell and the matrix are assumed to be isotropic (scalar).
We then consider the following conductivity problem:
\beq\label{main}
\begin{cases}
\nabla\cdot\Gs\nabla u=0 \quad &\mbox{in }\Rbb^d,\\
u(x)-a \cdot x=O(|x|^{-d+1})\quad&\mbox{as }|x|\rightarrow\infty,
\end{cases}
\eeq
where $a$ is a unit vector representing the background uniform field.

In absence of the inclusion $(D,\GO)$ the field is uniform, i.e., $\nabla u =a$. This uniform field is perturbed by insertion of the inclusion and the perturbation is not zero in general. It is known that the solution $u$ to \eqref{main}, or the perturbation $u-a\cdot x$, admits the following dipole asymptotic expansion:
\beq\label{farfield}
u(x)-a\cdot x = \frac{1}{\Go_d} \frac{\la Ma, x \ra}{|x|^d} + O(|x|^{-d}), \quad |x| \to \infty,
\eeq
where $\Go_d$ is the surface area of $S^{d-1}$, the $(d-1)$-dimensional sphere, and $M$ is a $d \times d$ matrix and called the polarization tensor (PT), which is determined by the inclusion $(D, \GO)$ and conductivity ratios $(\Gs_c/\Gs_m, \Gs_s/\Gs_m)$, namely,
$$
M=M(D, \GO)=M(D,\GO;\Gs_c/\Gs_m, \Gs_s/\Gs_m)
$$
(see, for example, \cite{AmKa07Book2, mbook}). The question of the PT-vanishing structure can be rephrased as follows: Given $D$ of arbitrary shape, find $\GO \supset \ol{D}$ so that $M(D, \GO)=0$.

If $D$ is a disk or a ball, then one can choose $\GO$ to be a concentric disk or ball and the conductivity parameters so that the perturbation $\nabla u - a$ of the uniform field $a$ is zero outside $\GO$. In fact, if $D=\{|x|<r_i\}$ and $\GO=\{|x|<r_e\}$ in $\Rbb^3$, and if the following relation among conductivities and the volume fractions holds:
\beq\label{neutral}
(2\Gs_s + \Gs_c)(\Gs_m-\Gs_s)+ \rho^3 (\Gs_s-\Gs_c)(2\Gs_s+\Gs_m)=0,
\eeq
where $\rho=r_i/r_e$ and $\rho^3$ is the volume fraction, then the solution $u$ to \eqnref{main} satisfies
\beq\label{neutral2}
u(x)-a \cdot x \equiv 0 \quad\mbox{for all } x \in \Rbb^3\setminus \GO.
\eeq
This discovery of Hashin and Shtrikman \cite{H1, H2} has laid significant implications in the theory of composite for which we refer to \cite{mbook}.

The inclusion $(D,\GO)$ is said to be neutral to multiple uniform fields if \eqnref{neutral2} holds for all constant vector $a$. However, a pair of concentric balls is the only structure neutral to multiple uniform fields as proved in \cite{KLS3d}. It is worth mentioning that the problem \eqnref{main} is well-posed even if $\Gs_m$ is a positive-definite matrix. It is believed to be true, but has not been proved, that if $\Gs_m$ is a positive-definite matrix, then the only inclusion neutral to multiple uniform fields is a pair of confocal ellipsoids whose common foci are determined by the eigenvalues of $\Gs_m$. See \cite{KLS3d} for descriptions of this problem and a related over-determined problem (see also \cite{JKLS}). The question in two dimensions has been solved \cite{KL2d, M.S}.

While the problem \eqnref{main} requires $u(x)-a \cdot x=O(|x|^{-d+1})$ at $\infty$ and the neutrality requires \eqnref{neutral2}, the PT-vanishing property requires in-between them, namely,
\beq\label{weakneutral}
u(x)-a \cdot x=O(|x|^{-d})\quad\mbox{as }|x|\rightarrow\infty,
\eeq
as one can see from \eqnref{farfield}. This is the reason why the PT-vanishing inclusion is also called the weakly neutral inclusion, as used in the title of the earlier version of the manuscript (arXiv:1911.07250v1). However, the name `PT-vanishing structure' seems to convey the meaning more directly, and the title has been changed accordingly in this new version of the manuscript. 

To present the main result of this paper in a precise manner, let $S^2$ be the unit sphere in $\Rbb^3$ and let $W^{2,\infty}(S^2)$ be the collection of all functions $f$ on $S^2$ such that
$$
\| f \|_{2,\infty}: = \| f \|_\infty + \| \nabla_T f \|_\infty + \| \nabla_T^2 f \|_\infty < \infty,
$$
where $\nabla_T$ and $\nabla_T^2$ be tangential gradient and Hessian on $S^2$, and $\| \cdot \|_p$ denotes the usual $L^p$ norm. The space
$W^{1,\infty}(S^2)$ with norm $\| \cdot \|_{1,\infty}$  is defined similarly.

Let $D_0: =\{|x|<r_i\}$ for some radius $r_i$. The core in this paper is defined to be a perturbation of $D_0$ by a function $h \in W^{2,\infty}(S^2)$. Denoting it by $D_h$, it is defined by
\beq\label{Dh}
\p D_h=\left\{ ~x ~|~x =(r_i+ h(\hat{x}))\hat{x}, \quad |\hat{x}|=1 ~ \right\}.
\eeq
The shell is defined also to be a perturbation of a ball. Let $\GO_0:= \{|x|<r_e\}$ ($r_e>r_i$) and define its perturbation by
\beq\label{GOb}
\p \GO_b =\left\{ ~x ~|~x =(r_e + b(\hat{x})) \hat{x}, \quad |\hat{x}|=1 ~ \right\}.
\eeq
The perturbation function $b$ for the shell is chosen from a subclass of $W^{2,\infty}(S^2)$: Let
\beq\label{Yl}
\{Y_l\}_{l=1}^6:=\left\{ \frac{1}{\sqrt{15}},  \, \hat{x}_1\hat{x}_2, \, \hat{x}_2\hat{x}_3, \, \frac{1}{2\sqrt{3}}  (-\hat{x}_1^2-\hat{x}_2^2+2\hat{x}_3^2), \, \hat{x}_1\hat{x}_3, \, \frac{1}{2} (\hat{x}_1^2-\hat{x}_2^2) \right\}.
\eeq
We mention that $Y_1$ is constant (a spherical harmonics of order 0) and $Y_l$, $2 \le l \le 6$, is a spherical harmonics of order 2, and $Y_l$, $1 \le l \le 6$, are mutually orthogonal and normalized so that the following holds for all $l$:
\beq\label{normalize}
\int_{S^2} |Y_l|^2 dS = \frac{4\pi}{15}.
\eeq
We take this normalization just for ease of notation. Let $W_6$ be the space spanned by $\{ Y_l \}$ and $\GO_b$ is defined for $b \in W_6$.

If $h$ and $b$ are sufficiently small, then $\ol{D_h} \subset \GO_b$ and hence the PT corresponding to $(D_h, \GO_b)$, which is denoted by $M=M(h,b)$, is well-defined. We choose $r_e$, the radius of $\GO_0$, so that $r_i$ and $r_e$ satisfy neutrality condition \eqnref{neutral} for given conductivities $\Gs_c$, $\Gs_s$ and $\Gs_m$. For that, $\Gs_c$, $\Gs_s$ and $\Gs_m$ need to satisfy
\beq\label{neutral3}
0 < \frac{(2\Gs_s + \Gs_c)(\Gs_s-\Gs_m)}{(2\Gs_s+\Gs_m)(\Gs_s-\Gs_c)} <1.
\eeq
Then $(D_0,\GO_0)$ is neutral, namely, $M(0,0)=0$.

The following is the main result of this paper:
\begin{thm}\label{main_thm}
Given $r_i$, let $r_e$ satisfy the neutrality condition \eqnref{neutral}.
There is $\Gve>0$ such that for each $h \in W^{2,\infty}(S^2)$ with $\| h \|_{2,\infty}<\Gve$ there is $b=b(h) \in W_6$ such that
\beq
M(h, b(h))=0,
\eeq
namely, the inclusion $(D_h, \GO_{b(h)})$ of the core-shell structure is PT-vanishing. The mapping $h \mapsto b(h)$ is continuous.
\end{thm}

Let us briefly describe how Theorem \ref{main_thm} is proved. Since $M=M(h,b)=(m_{ij})_{i,j=1}^3$ is a symmetric matrix, we can identify $M$ with $(m_{11}, m_{22}, m_{33},  m_{12}, m_{13}, m_{23})$. We then regard $M$ as a function from $U \times V$ into $\Rbb^6$, where $U$ is a small neighborhood of $0$ in $W^{2,\infty}(S^2)$ and $V$ is a small neighborhood of $0$ in $W_6$ identified with $\Rbb^6$, i.e.,
$$
M: U \times V \subset W^{2,\infty}(S^2) \times \Rbb^6 \to \Rbb^6.
$$
Moreover, since $(D_0, \GO_0)$ is neutral to multiple fields, it is PT-vanishing, namely, $M(0,0)=0$. We then show the Jacobian determinant of $M$ is non-zero, namely,
\beq\label{bJacob}
\frac{\p (m_{11}, m_{22}, m_{33},  m_{12}, m_{13}, m_{23})}{\p (b_1,b_2, b_3,b_{4},b_{5},b_6)} (0,0) \neq 0.
\eeq
Then, Theorem \ref{main_thm} follows from the implicit function theorem (Theorem \ref{thm:ift}).

The idea and structure of the proof are the same as those in \cite{KLS2D}. However, since we are dealing with spherical harmonics in three dimensions in this paper, details are much more involved.

By switching roles of $h$ and $b$, we obtain the following theorem:
\begin{thm}\label{main_thm2}
Given $r_e$, let $r_i$ satisfy the neutrality condition \eqnref{neutral}.
There is $\Gve>0$ such that for each $h \in W^{2,\infty}(S^2)$ with $\| h \|_{2,\infty}<\Gve$ there is $b=b(h) \in W_6$ such that
the inclusion $(D_{b(h)}, \GO_h)$ of the core-shell structure is PT-vanishing. The mapping $h \mapsto b(h)$ is continuous.
\end{thm}

This paper is organized as follows. In section \ref{sec:integral}, we review the definition of the PT in terms of a system of integral equations, and prove continuity and differentiability of the relevant integral operator. Section \ref{sec:com} includes some preliminary computations of quantities to be used in proving Theorem \ref{main_thm}, which is proved in section \ref{sec:proof}. This paper ends with a short conclusion.

\section{The integral equations and its stability properties}\label{sec:integral}

\subsection{Preliminary: layer potentials and PT}\label{subsec:PT}

Let $G(x)$ be the fundamental solution to the Laplacian, that is, $G(x)= 1/(2\pi) \log |x|$ in two dimensions, and $G(x)=- (4\pi |x|)^{-1}$ in three dimensions. Let $D$ be a bounded simply connected domain with Lipschitz continuous boundary. Let $\Scal_{\p D}$ and $\Kcal_{\p D}^*$ be the single layer potential and the Neumann-Poincar\'e operator, respectively, namely, for a function $\Gvf \in L^2(\p D)$
\beq
\Scal_{\p D}[\Gvf](x) := \int_{\p D} G(x-y) \Gvf(y)\, dS(y), \quad x\in \Rbb^3,
\eeq
and
\beq
\Kcal_{\p D}^*[\Gvf](x)= \int_{\p D} \p_{\nu_x} G(x-y) \Gvf(y)\, dS(y),
\eeq
where $dS$ is the surface element on $\p D$ and $\p_{\nu}$ denotes the outward normal derivative on $\p D$. The relation between $\Scal_{\p D}$ and $\Kcal_{\p D}^*$ is given by the following jump formula:
\beq
\p_\nu \Scal_{\p D}[\Gvf](x)\big|_{\pm} = \left( \pm \ds\frac{1}{2}I+\Kcal_{\p D}^{*} \right)[\Gvf](x), \quad\mbox{a.e. } x\in\p D,
\eeq
where $I$ is the identity operator and subscripts $\pm$ denote the limits from outside and inside $D$, respectively.

Let $\GO$ and $D$ be two bounded domains such that $\ol{D} \subset \GO \subset \Rbb^d$, whose boundaries are assumed to be Lipschtiz continuous.  The solution $u_l$ ($1 \le l \le d$) to \eqnref{main} when $a \cdot x= x_l$ is represented as
\beq
u_l(x)=x_l+\SD[\Gvf_1^{(l)}](x)+\SO[\Gvf_2^{(l)}](x),\quad x\in\Rbb^d,
\eeq
where $(\Gvf_1^{(l)},\Gvf_2^{(l)})\in L^2_0(\p D)\times  L^2_0(\p \GO)$ is the unique solution to the system of integral equations
\beq\label{pp}
\begin{bmatrix}
-\Gl I+\KDS & \p_\nu\SO\\
\p_\nu\SD & -\mu I+\KOS
\end{bmatrix}
\begin{bmatrix}
\Gvf_1^{(l)}\\
\Gvf_2^{(l)}
\end{bmatrix}
=-\begin{bmatrix}
\nu^l_{\p D}\\
\nu^l_{\p \GO}
\end{bmatrix}.
\eeq
Here $\nu^l_{\p D}$ is the $l$-th component of the outward unit normal vector $\nu_{\p D}$ to $\p D$, $\nu^l_{\p\GO}$ is defined likewise, and the numbers $\Gl$ and $\Gm$ are given by
\beq\label{lambdamu}
\Gl=\frac{\Gs_c+\Gs_s}{2(\Gs_c-\Gs_s)}\quad\mbox{and}\quad \mu=\frac{\Gs_s+\Gs_m}{2(\Gs_s-\Gs_m)}.
\eeq
Here and afterwards, $L^2_0(\p D)$ denotes the collection of square integrable functions on $\p D$ with the mean zero. We refer to the discussion in \cite{KLS2D} for a proof of unique solvability of \eqnref{pp} on $L^2_0(\p D)\times  L^2_0(\p \GO)$.

The PT $M=M(D, \GO)=(m_{ll'})_{l,l'=1}^d$ of the core-shell structure $(D,\GO)$ is defined by
\beq\label{PT0}
m_{ll'}=\int_{\p D}x_{l'}\Gvf_1^{(l)} \, dS+\int_{\p \GO}x_{l'} \Gvf_2^{(l)} \, dS, \quad l,l'=1,\ldots,d.
\eeq
The expansion \eqnref{farfield} of the solution $u$ to \eqnref{main} is valid with this PT.

\subsection{Parametrizations of integral equations}

We consider the system of integral equations \eqnref{pp} when $D=D_h$ and $\GO=\GO_b$ where $D_h$ and $\GO_b$ are defined by \eqnref{Dh} and \eqnref{GOb}, respectively:
\begin{equation}\label{pp1}
\begin{cases}
\ds \left( -\Gl I+\Kcal^*_{\p D_h} \right) [\Gvf_1] + \p_\nu\Scal_{\p\GO_b} [\Gvf_2] = \psi_1 \quad &\mbox{on } \p D_h ,\\
\ds \p_\nu\Scal_{\p D_h}[\Gvf_1] + \left( -\Gm I+\Kcal^*_{\p\GO_b} \right) [\Gvf_2] = \psi_2 \quad &\mbox{on } \p\GO_b ,
\end{cases}
\end{equation}
on $L^2_0(\p D_h) \times L^2_0(\p\GO_b)$. This system of equations admits a unique solution and there is a constant $C=C(h,b)$ such that
\beq\label{bounded}
\| \Gvf_1 \|_{L^2(\p D_h)} + \| \Gvf_2 \|_{L^2(\p\GO_b)} \le C (\| \psi_1 \|_{L^2(\p D_h)} + \| \psi_2 \|_{L^2(\p\GO_b)}).
\eeq

We now transform \eqnref{pp1} in three dimensions to a system of integral equations on $L^2_0(S^2)^2$ where $S^2$ is the unit sphere. To do so, let
\beq\label{icoordi}
x_{i,h}(x):=(r_i+ h(x))x, \quad x\in S^2,
\eeq
which is a change of variables from $S^2$ onto $\p D_h$.
Then the unit normal vector $\nu(x_{i,h}(x))=:\nu_{i,h}(x)$ on $\p D_h$ is given by the relation
\beq\label{inormal}
J_{i,h}(x) \nu_{i,h}(x) =  (r_i + h(x)) \big[ ( r_i+h(x))x - \nabla_T h(x) \big],
\eeq
where
\beq\label{iJaco}
J_{i,h}(x):=(r_i + h(x))\sqrt{(r_i+h(x))^2  + |\nabla_T h(x)|^2}.
\eeq
The tangential gradient $\nabla_T h(x)$, which was already used in Introduction, is defined to be
\beq\label{tangent}
\nabla_T h(x) = \sum_{j=1}^2 \la \nabla h(x), T_j(x) \ra T_j(x),
\eeq
where $T_1$ and $T_2$ are two unit orthogonal tangent vector fields on $S^2$, and $\nabla h$ is defined after extending $h$ to a tubular neighborhood of $S^2$. Note that $J_{i,h}(x)$ is the Jacobian determinant of the change of variables $x_{i,h}(x)$, namely, the following formula holds:
\beq\label{ichange}
dS(x_{i,h}(x)) = J_{i,h}(x) dS(x), \quad x \in S^2.
\eeq

Likewise, let
\beq\label{ecoordi}
x_{e,b}(x):=(r_e+ b(x))x, \quad x\in S^2,
\eeq
which is a change of variables from $S^2$ onto $\p \GO_b$.
Then the normal vector $\nu(x_{e,b}(x))=:\nu_{e,b}(x)$ on $\p\GO_b$ is given by
\beq\label{enormal}
J_{e,b}(x) \nu_{e,b}(x) =  (r_e + b(x)) \big[ ( r_e+b(x))x - \nabla_T b(x) \big],
\eeq
where
\beq\label{eJaco}
J_{e,b}(x):=(r_e + b(x))\sqrt{(r_e +b(x))^2  + |\nabla_T b(x)|^2}.
\eeq
Then it holds that
\beq\label{echange}
dS(x_{e,h}(x)) = J_{e,h}(x) dS(x), \quad x \in S^2.
\eeq

Straight-forward calculations using \eqnref{icoordi}-\eqnref{echange} show that the following relations hold:
\begin{itemize}
\item Let $A(h)$ be the operator on $L^2(S^2)$ defined by the integral kernel
\beq\label{Ah}
A_h(x,y)  =\frac{1}{4\pi}\frac{\la x_{i,h}(x)-x_{i,h}(y),\nu_{i,h}(x)  \ra}{|x_{i,h}(x)-x_{i,h}(y)|^3} J_{i,h}(x).
\eeq
Then,
\beq\label{Arel}
A(h)[f_1](x) = J_{i,h}(x) \Kcal^*_{\p D_h} [\Gvf_1](x_{i,h}(x)),
\eeq
where $f_1(y):= \Gvf_1(x_{i,h}(y)) J_{i,h}(y)$.

\item Let $B(b)$ be defined by
\beq\label{keb_ker}
B_b(x,y)=\frac{1}{4\pi}\frac{\la x_{e,b}(x)-x_{e,b}(y),\nu_{e,b}(x)  \ra}{|x_{e,b}(x)-x_{e,b}(y)|^3} J_{e,b}(x).
\eeq
Then,
\beq\label{Brel}
B(b)[f_2](x) = J_{e,b}(x) \Kcal^*_{\p \GO_b} [\Gvf_2](x_{e,b}(x)),
\eeq
where $f_2(y):= \Gvf_2(x_{e,b}(y)) J_{e,b}(y)$.

\item Let $C(h,b)$ be defined by
\beq\label{C}
C_{h,b} (x,y)=\frac{1}{4\pi}\frac{\la x_{i,h}(x) - x_{e,b}(y),  \nu_{i,h}(x)  \ra}{|x_{i,h}(x) - x_{e,b}(y)|^3} J_{i,h}(x).
\eeq
Then,
\beq\label{Crel}
C(h,b)[f_2](x) = J_{i,h}(x) \p_\nu\Scal_{\p\GO_b} [\Gvf_2](x_{i,h}(x)) .
\eeq

\item Let $D(h,b)$ be defined by
\beq\label{Dkernel}
D_{h,b} (x,y)=\frac{1}{4\pi}\frac{\la x_{e,b}(x) - x_{i,h}(y),  \nu_{e,b}(x)  \ra}{|x_{e,b}(x) - x_{i,h}(y)|^3} J_{e,b}(x).
\eeq
Then,
\beq\label{Drel}
D(h,b)[f_1](x)= J_{e,b}(x) \p_\nu\Scal_{\p D_h} [\Gvf_1](x_{e,b}(x)) .
\eeq
\end{itemize}

Thanks to above formulae, the integral equation \eqnref{pp1} now takes the form
\beq\label{system}
\begin{cases}
\ds \left( -\Gl I+ A(h) \right) [f_1] + C(h,b) [f_2]  = g_1 , \\
\ds D(h,b) [f_1] + \left( -\Gm I+ B(b) \right) [f_2]  = g_2 ,
\end{cases}
\eeq
where
\beq
g_1(x):= J_{i,h}(x) \psi_1 (x_{i,h}(x)) \quad\mbox{and}\quad g_2(x):= J_{e,b}(x) \psi_2 (x_{e,b}(x)).
\eeq
Let
\beq
\Acal(h,b) :=
\begin{bmatrix}
-\Gl I+ A(h) & C(h,b) \\
D(h,b) & -\Gm I+ B(b)
\end{bmatrix}
\eeq
and $f=(f_1,f_2)^\top$, $g=(g_1,g_2)^\top$ ($\top$ for transpose). Then \eqnref{system} can be written in short as
\beq\label{sys}
\Acal(h,b)f=g
\eeq
on $L^2_0(S^2)^2$. Moreover, \eqnref{bounded} shows that there is a constant $K > 0$ depending on $h$ and $b$ such that
$$
\| \Acal(h,b)^{-1} g \|_{2}\le K \| g \|_{2} ,
$$
where $\| \cdot \|_{2}$ denotes the norm on $L^2(S^2)^2$.  Here and throughout this paper $K$ denotes a positive constant which may differ at each appearance.

\subsection{Continuity of the integral operator}

We now consider the continuity of the operator $\Acal(h,b)$ with respect to $h$ and $b$. For that we assume that $b \in W^{2,\infty}(S^2)$. We first obtain the following proposition for the continuity. This proposition for two dimensions was obtained in \cite{KLS2D}. Even though the idea and the procedure of the proof are almost identical, the proof here and there are technically dissimilar because of the nature of the integral kernels. For example, the integral kernels $A_h(x,y)$ and $B_b(x,y)$ have singularities of order 1 at $x=y$ in three dimensions, while there is no singularity in two dimensions. See the first paragraph of the following proof.

\begin{prop}\label{prop:stab}
There is $\Gve >0$ such that if $h, b \in W^{2,\infty}(S^2)$ and $\| h \|_{2,\infty} + \| b \|_{2,\infty} \le \Gve$, then
\begin{itemize}
\item[(i)] $\Acal(h,b)$ is continuous at $(h,b)=(0,0)$ strongly, namely, there is a constant $K > 0$ such that
\beq\label{contione}
\| (\Acal(h,b)- \Acal(0,0)) [f] \|_2 \le K(\| h \|_{2,\infty} + \| b \|_{2,\infty}) \| f \|_2
\eeq
for all $f \in L^2(S^2)^2$.
\item[(ii)] $\Acal(h,b)$ is continuous at $(h,b) \neq (0,0)$ weakly, namely, for each $f \in L^2(S^2)^2$
\beq\label{contitwo}
\| (\Acal(k,d)- \Acal(h,b)) [f] \|_2 \to 0
\eeq
as $\| k-h \|_{2,\infty} + \| d-b \|_{2,\infty} \to 0$.
\end{itemize}
\end{prop}

\begin{proof}
We first deal with the operator $B(b)$ and include the proof in detail here since the proof is more involved than that for the two-dimensional case in \cite{KLS2D} due to the singularity. The operator $A(h)$ can be treated similarly. The operators $C(h,b)$ and $D(h,b)$ can be dealt with in the same way as in two dimensions since their integral kernels do not have singularities. However, we include a brief proof since the details presented in this proof will be used in the later part of the paper.

One can easily see from \eqnref{ecoordi} and \eqnref{enormal} that
\begin{align*}
& \big\la x_{e,b}(x)-x_{e,b}(y) , J_{e,b}(x) \nu_{e,b}(x)  \big\ra \\
& = (r_e+b(x))^3 - (r_e+b(x)) (r_e+b(y)) \big[ \left( r_e + b(x) \right) x\cdot y -y\cdot \nabla _T b(x) \big].
\end{align*}
Here, we used the fact that $x \cdot \nabla _T b(x) =0$. Therefore, we have
\beq\label{expnum}
\big\la x_{e,b}(x)-x_{e,b}(y) ,J_{e,b}(x) \nu_{e,b}(x) \big\ra
=\frac{1}{2} r_e^3 |x-y|^2 \big( 1 + R_1 \big),
\eeq
where
\begin{multline}\label{Rone}
R_1=R_1(b;x,y)= \frac{2b(x)+b(y)}{r_e} + \frac{b(x)(b(x)+2b(y))}{r_e^2} + \frac{b^2(x)b(y)}{r_e^3}\\
+\frac{2(r_e+b(x))(b(x)-b(y))^2+2(r_e+b(x))(r_e+b(y)) \big[ b(x)-b(y) + y\cdot \nabla_Tb(x) \big]}{r_e^3 |x-y|^2}.
\end{multline}
Let $b$ be extended to $\Rbb^3 \setminus \{0\}$ by defining $b(x)=b(x/|x|)$. Then, since
$$
 b(x)-b(y) + y\cdot \nabla_T b(x) = -b(y) + b(x) + (y-x)\cdot \nabla b(x)
$$
for $x, y \in S^2$, we have from Taylor's theorem
\beq\label{taylor}
|b(x)-b(y) + y\cdot \nabla_T b(x)| \le K \| b \|_{2,\infty}|x-y|^2
\eeq
for some constant $K$. Thus,
\begin{align}
&\big| (r_e+b(x))(b(x)-b(y))^2+ (r_e+b(x))(r_e+b(y))[b(x)-b(y) + y\cdot \nabla_Tb(x) ] \big| \nonumber \\
&\le K \| b \|_{2,\infty} |x-y|^2.
\label{use of Taylor formula 2}
\end{align}
The constant $K$ may differ at each occurrence. We then infer
\beq\label{Rone2}
\sup_{x,y \in S^2} |R_1(b; x,y)| \le K \| b \|_{2,\infty}.
\eeq

One can also see that
$$
\left| x_{e,b}(x)-x_{e,b}(y) \right|^2
= r_e^2 |x-y|^2 + |b(x)x - b(y)y|^2 + 2 r_e (x-y)\cdot (b(x)x - b(y)y).
$$
Thus we have
\beq\label{expden}
\left| x_{e,b}(x)-x_{e,b}(y) \right|^3 = r_e^3 |x-y|^3 (1 + R_2)^{3/2},
\eeq
where
\beq\label{Rtwo}
R_2= R_2(b;x,y) = \frac{2(x-y)\cdot(b(x)x - b(y)y)}{r_e |x-y|^2 } + \frac{|b(x)x - b(y)y|^2}{r_e^2 |x-y|^2}.
\eeq
Note that
\beq\label{Rtwo2}
\sup_{x,y \in S^2} |R_2 (b;x,y)| \le K \|b \|_{1,\infty} .
\eeq

We have from \eqnref{keb_ker}, \eqnref{expnum} and \eqnref{expden} that
\beq\label{BRoneRtwo}
B_b (x,y)=\frac{1}{8\pi}\frac{1}{|x-y|}\frac{1}{\sqrt{1+ R_2 }}\left[ 1+ \frac{R_1- R_2}{1 + R_2} \right].
\eeq
The singularity $|x-y|^{-1}$ on the righthand side above is specific to three dimensions, and does not appear in two dimensions. Since
\begin{align*}
B_b (x,y) - B_0 (x,y) &=\frac{1}{8\pi}\frac{1}{|x-y|}\frac{1}{\sqrt{1+ R_2 }}\left[ 1+ \frac{R_1- R_2}{1 + R_2} - \sqrt{1+ R_2 } \right] \\
&= \frac{1}{8\pi}\frac{1}{|x-y|}\frac{1}{\sqrt{1+ R_2 }}\left[  \frac{R_1- R_2}{1 + R_2} - \frac{R_2 }{1+\sqrt{1+ R_2 } }\right] ,
\end{align*}
one can see from \eqnref{Rone2} and \eqnref{Rtwo2} that
\beq\label{differ}
|B_b (x,y) - B_0 (x,y)| \le \frac{K \|b\|_{2,\infty}}{|x-y|},
\eeq
provided that $\|b \|_{1,\infty}$ is sufficiently small (because of \eqnref{Rtwo2}). Thus we have
\beq\label{Bone}
\| (B(b)-B(0)) [f_2] \|_2 \le K \|b\|_{2,\infty} \| f_2 \|_2
\eeq
for all $f_2 \in L^2(S^2)$.

Let $f_2 \in L^2(S^2)$ and $\Gd$ be an arbitrary but fixed positive small number. For each $x \in S^2$, we write
\begin{align*}
(B(d)-B(b)) [f_2](x) &= \int_{S^2} (B_d (x,y) - B_b (x,y)) f_2(y) \, dS(y) \\
& = \int_{|y-x| \le \Gd} + \int_{|y-x| > \Gd} =: I_\Gd(x) + II_\Gd(x).
\end{align*}
If $\| d -b \|_{2, \infty} \to 0$, then $B_d (x,y) - B_b (x,y) \to 0$ unless $x =y$. Moreover, we have from \eqnref{differ}
$$
\big| (B_d (x,y) - B_b (x,y)) f_2(y) \big| \le K \frac{|f_2(y)|}{|x-y|} \le \frac{K}{\Gd} |f_2(y)| ,
$$
provided that $|x-y| > \Gd$.
Thus, by Lebesgue dominated convergence theorem, we infer that $II_\Gd(x) \to 0$ for each $x$ as $d \to b$ in $W^{2, \infty}(S^2)$. Further, we have
$$
|II_\Gd(x)| \le \frac{C}{\Gd} \| f_2 \|_2.
$$
We then apply Lebesgue dominated convergence theorem once more to infer that $\| II_\Gd \|_2 \to 0$ as $d \to b$ in $W^{2, \infty}(S^2)$.

To handle $I_\Gd(x)$, we first observe that
$$
|I_\Gd(x)| \le C \int_{|y-x| \le \Gd} \frac{|f_2(y)|}{|y-x|} dS(y) = C \sum_{j=1}^\infty \int_{2^{-j} \Gd < |y-x| \le 2^{-j+1} \Gd} \frac{|f_2(y)|}{|y-x|} dS(y).
$$
Thus we have
\begin{align*}
|I_\Gd(x)| &\le C \sum_{j=1}^\infty \frac{1}{2^{-j} \Gd} \int_{2^{-j} \Gd < |y-x| \le 2^{-j+1} \Gd} |f_2(y)| dS(y) \\
&\le C \Gd \sum_{j=1}^\infty \frac{2^{-j}}{2^{-2j} \Gd^2} \int_{|y-x| \le 2^{-j+1} \Gd} |f_2(y)| dS(y) \le C\Gd M(|f_2|)(x),
\end{align*}
where $M$ is the maximal function, namely,
$$
M(|f_2|)(x) = \sup_{r} \frac{1}{r^2} \int_{|y-x| \le r} |f_2(y)| dS(y)
$$
(up to some constant multiplication). Note that in the above the constant $C$ differs at each occurrence. It then follows that
$$
\int_{S^2} |I_\Gd(x)|^2 dS \le C\Gd^2 \int_{S^2} |M(|f_2|)(x)|^2 dS \le C\Gd^2 \|f_2 \|_2^2,
$$
where the last inequality comes from the fact that the maximal function $M$ is bounded on $L^2$, for which we refer to \cite{Stein-book-70}.

So far, we have shown that
$$
\| (B(d)-B(b)) [f_2] \|_2 \le C\Gd \|f_2 \|_2 + \| II_\Gd \|_2.
$$
Since $\lim_{d \to b} \| II_\Gd \|_2=0$,
$$
\limsup_{d \to b} \| (B(d)-B(b)) [f_2] \|_2 \le C\Gd \|f_2 \|_2 ,
$$
where $d \to b$ in $W^{2,\infty}(S^2)$. Since $\Gd$ is arbitrary, we conclude that $\| (B(d)-B(b)) [f_2] \|_2 \to 0$
for each fixed $f_2$.

Similarly, one can show that
\beq\label{Aone}
\| (A(h)-A(0)) [f_1] \|_2 \le K \|h\|_{2,\infty} \|f_1 \|_2
\eeq
for all $f_1 \in L^2(S^2)$, and
\beq\label{Atwo}
\| (A(k)-A(h)) [f_1] \|_2 \to 0
\eeq
as $k \to h$ in $W^{2,\infty}(S^2)$ for each fixed $f_1$.

To handle the operator $C(h,b)$, let
\beq\label{alpha}
\Ga(h,b;x,y):= \big\la x_{i,h}(x) - x_{e,b}(y),  J_{i,h}(x) \nu_{i,h}(x) \big\ra
\eeq
and
\beq\label{beta}
\Gb(h,b;x,y):= |x_{i,h}(x) - x_{e,b}(y)|^3
\eeq
so that
\beq\label{GaoverGb}
C_{h,b}(x,y)= \frac{\Ga(h,b;x,y)}{4\pi \Gb(h,b;x,y)}.
\eeq
One can see that
\beq\label{fisrtone}
\sup_{x,y \in S^2} | \Ga(h,b;x,y)- \Ga(k,d;x,y)| \le K(\| h-k \|_{1,\infty} + \| b-d \|_{1,\infty})
\eeq
and
\beq\label{secondone}
\sup_{x,y \in S^2} | \Gb(h,b;x,y)- \Gb(k,d;x,y)| \le K(\| h-k \|_{\infty} + \| b-d \|_{\infty}).
\eeq
In fact, it is straight-forward to derive \eqnref{fisrtone}. To show \eqnref{secondone}, we see that
$$
| \Gb(h,b;x,y)^{2/3}- \Gb(k,d;x,y)^{2/3}| \le K(\| h-k \|_{\infty} + \| b-d \|_{\infty})
$$
for all $x, y \in S^2$. Furthermore, we have
\beq\label{lower}
\Gb(h,b;x,y)= |x_{i,h}(x)-x_{e,b}(y)|^3 \ge \frac{1}{ 8} (r_e-r_i)^3,
\eeq
provided that $\| h \|_\infty$ and $\| b \|_\infty$ are sufficiently small. Thus we have \eqnref{secondone} by using a simple identity for positive numbers $a$ and $b$:
$$
a-b = \frac{(a^{2/3} - b^{2/3})(a^{4/3}+a^{2/3}b^{2/3} + b^{4/3})}{a+b}.
$$
It then follows from \eqnref{fisrtone}, \eqnref{secondone} and \eqnref{lower} that
$$
\sup_{x,y \in S^2} | C_{h,b}(x,y)- C_{k,d}(x,y)| \le K(\| h-k \|_{1,\infty} + \| b-d \|_{1,\infty}),
$$
from which we conclude that
\beq\label{Conetwo}
\| (C(h,b)-C(k,d)) [f_2] \|_2 \le K (\| h-k \|_{1,\infty} + \| b-d \|_{1,\infty}) \| f_2 \|_2
\eeq
for all $f_2 \in L^2(S^2)$.

Similarly one can show that
\beq\label{Donetwo}
\| (D(h,b)-D(k,d)) [f_1] \|_2 \le K (\| h-k \|_{1,\infty} + \| b-d \|_{1,\infty}) \| f_1 \|_2
\eeq
for all $f_1 \in L^2(S^2)$.
Now \eqnref{contione} and \eqnref{contitwo} follow, and the proof is complete.
\end{proof}

Note that $\Acal(0,0)$ is nothing but the operator appearing in \eqnref{pp}, and so it is invertible. Note also that
$$
\Acal(h,b)^{-1} = \left(I + \Acal(0,0)^{-1} \big( \Acal(h,b) - \Acal(0,0) \big) \right)^{-1} \Acal(0,0)^{-1}.
$$
Thanks to \eqnref{contione}, the operator norm of $\Acal(0,0)^{-1} ( \Acal(h,b) - \Acal(0,0))$ is small if $\|h\|_{2,\infty} + \|b\|_{2,\infty}$ is small. Thus $(I + \Acal(0,0)^{-1} ( \Acal(h,b) - \Acal(0,0) ))^{-1}$ exists. So, we have the following corollary.

\begin{cor}\label{cor}
There is $\Gve >0$ such that
$$
\| \Acal(h,b)^{-1} [g] \|_{2} \le K \| g \|_{2}
$$
for all $g \in L_0^2(S^2)^2$ for some $K > 0$ independent of $h$ and $b$ satisfying $\|h\|_{2,\infty} + \|b\|_{2,\infty} < \Gve $.
\end{cor}

\subsection{Differentiability of the integral operator}

We now look into differentiability of $\Acal(h,b)$ with respect to $b$ when $b$  belongs to $W_6$, namely, $b$ is of the form
\beq\label{b0}
b = \sum_{l=1}^{6} b_{l} Y_{l},
\eeq
where $Y_l$ is given by \eqnref{Yl}. For such a $b$, $\| b \|_{2,\infty}$ is equivalent to
$$
|b|_\infty := \max_{1 \le j \le 6} |b_j|.
$$
For the rest of this paper, we assume that $b$ is of the form \eqnref{b0}.

Let $\p_j$ denote the partial derivative with respect to $b_j$ ($j=1,\dots,6$). Since $\p_j b= Y_j$, $\p_j x_{e,b}(y)= Y_j(y)y$. Thus we see from the definitions \eqnref{alpha} and \eqnref{beta} that
$$
\p_j \Ga(h,b;x,y):= - \big\la Y_j(y)y,  J_{i,h}(x) \nu_{i,h}(x) \big\ra
$$
and
$$
\p_j \Gb(h,b;x,y):= -3 Y_j(y) |x_{i,h}(x) - x_{e,b}(y)| (x_{i,h}(x) - x_{e,b}(y)) \cdot y.
$$
We then see easily from \eqnref{GaoverGb} and \eqnref{lower} that
$$
\sup_{x,y \in S^2} \left| \p_j C_{h,b}(x,y) \right|  \le K
$$
for some constant $K >0$.
Moreover, if $k \in W^{2,\infty}(S^2)$ and $d= \sum_{l=1}^{6} d_{l} Y_{l}$, then
$$
\sup_{x,y \in S^2} \left| \p_j C_{h,b}(x,y) - \p_j C_{k,d}(x,y) \right|  \le K (\| h-k \|_{1,\infty} + |b-d|_\infty).
$$
Thus we see that the operator $\p_j  C(h,b)$ is bounded on $L^2(S^2)$ and
\beq\label{Cderi}
\left\| \left( \p_j  C(h,b) - \p_j  C(k,d) \right) [f_2] \right \|_2 \le K (\| h-k \|_{1,\infty} + |b-d|_\infty) \| f_2 \|_2
\eeq
for all $f_2 \in L^2(S^2)$.

Similarly one can see that the operator $\p_j  D(h,b)$ is bounded on $L^2(S^2)$ and
\beq\label{Dderi}
\left\| \left( \p_j  D(h,b) - \p_j  D(k,d) \right) [f_1] \right \|_2 \le K (\| h-k \|_{1,\infty} + |b-d|_\infty) \| f_1 \|_2
\eeq
for all $f_1 \in L^2(S^2)$.

Let $R_1$ and $R_2$ be the quantities defined by \eqnref{Rone} and \eqnref{Rtwo}, respectively, with $b$ and $d$ of the form \eqnref{b0}. We claim that the following inequalities hold for $l=1,2$ and $j,k=1,\dots,6$:
\begin{align}
\sup_{x,y \in S^2} |\p_j R_l(b,x,y)| & \le K, \label{pR11} \\
\sup_{x,y \in S^2} | R_l(b,x,y)- R_l(d,x,y)| &\le K|b-d|_\infty,  \label{pR12} \\
\sup_{x,y \in S^2} |\p_k \p_j R_l(b,x,y)| &\le K,  \label{pR13} \\
\sup_{x,y \in S^2} |\p_j R_l(b,x,y) - \p_j R_l(d,x,y)| &\le K |b-d|_\infty . \label{pR14}
\end{align}
In fact, since $\p_j b= Y_j$, we obtain from \eqnref{Rone}
\begin{align}
\p_j R_1(b,x,y) &= \frac{2Y_j(x) + Y_j(y)}{r_e} + \frac{Y_j(x)(b(x)+ 2b(y))+b(x)(Y_j(x)+2Y_j(y))}{r_e^2} \nonumber \\
& \quad + \frac{2b(x)Y_j(x) b(y)+b^2(x)Y_j(y)}{r_e^3} + \frac{2}{r_e^3|x-y|^2} I_b(x,y), \label{pR1}
\end{align}
where
\begin{align*}
I_b(x,y) &:= Y_j(x)(b(x)-b(y))^2 + 2(r_e+b(x))(b(x)- b(y))(Y_j(x)- Y_j(y)) \\
& \quad + [r_e(Y_j(x)+Y_j(y))+Y_j(x)b(y)+b(x)Y_j(y)][b(x)-b(y) + y\cdot\nabla_Tb(x) ]  \\
& \quad +(r_e+b(x))(r_e+b(y))[Y_j(x)-Y_j(y) + y\cdot\nabla_T Y_j(x))] .
\end{align*}
Using \eqnref{taylor}, we see that
$$
|I_b(x,y)| \le K |x-y|^2
$$
for some $K$. Thus we arrive at
\beq
\sup_{x,y \in S^2} |\p_j R_1(b,x,y)| \leq K , \quad j=1,\dots,6.
\eeq
We then immediately obtain \eqnref{pR11}  and hence \eqref{pR12} when $l=1$. By taking further derivatives in \eqnref{pR1}, one can also show \eqref{pR13} and \eqref{pR14}. The case when $l=2$ can be proved similarly using \eqnref{Rtwo}.

From \eqref{BRoneRtwo} we have, for $j=1,\dots,6$,
\begin{align}
\p_j B_b(x,y) &= \frac{1}{8\pi} \frac{1}{|x-y|}\Big[ -\frac{1}{2}(1+ R_2 )^{-3/2} \p_j R_2  \left( 1+ \frac{R_1- R_2}{1 + R_2} \right) \nonumber \\
& \quad + \left( 1+ R_2  \right)^{-1/2} \left( \frac{\p_j R_1 - \p_j R_2}{1+R_2} -  \frac{ (R_1 - R_2) \p_j R_2}{(1+R_2)^2} \right)\Big]. \label{p1b}
\end{align}
It then follows from \eqref{Rone2}, \eqref{Rtwo2}, and \eqnref{pR11}-\eqnref{pR14} that for $j,k=1,\dots,6$
$$
\sup_{x,y \in S^2} |\p_k\p_j B_b (x,y) | \le K,
$$
and hence
$$
\sup_{x,y \in S^2} |\p_j B_b (x,y) - \p_j B_d (x,y) | \le K|b-d|_\infty.
$$
Thus we have
\beq\label{Bderi}
\left\| \left( \p_j  B(b) - \p_j  B(d) \right) [f_2] \right \|_2 \le K |b-d|_\infty \| f_2 \|_2
\eeq
for all $f_2 \in L^2(S^2)$.

The following proposition is an immediate consequence of \eqnref{Cderi}, \eqnref{Dderi} and \eqnref{Bderi}.
\begin{prop}\label{prop:deri}
There is a constant $\Gve >0$ such that if $b$ is of the form \eqnref{b0} and $\| h\|_{1,\infty} + |b|_\infty<\Gve$, then $\p_j \Acal(h,b)$ is bounded on $L^2(S^2)^2$ for $j=1,\dots,6$. Moreover, there is $K>0$ such that if $d$ is of the form \eqnref{b0} and $\| k \|_{1,\infty} + |d|_\infty <\Gve$, then
\beq\label{Aderi}
\left\| \left( \p_j  \Acal(h,b) - \p_j  \Acal(k,d) \right) [f] \right \|_2 \le K (\| h-k \|_{1,\infty} + |b-d|_\infty) \| f \|_2
\eeq
for all $f \in L^2(S^2)^2$.
\end{prop}

\section{Some computations}\label{sec:com}

In this section we compute the quantities
\beq\label{quantities}
\left\la  x_{l'} , \p_j B(0)[x_l] \right\ra, \quad \la x_{l'}, \p_j C(0,0)[x_l] + \Gr \p_j D (0,0)[x_l] \ra,
\eeq
for $1\leq l\leq l'\leq3$ and $j=1,\dots,6$. Here $\la \ , \ \ra$ denotes the inner product on $L^2(S^2)$ and $\Gr=r_i/r_e$. These quantities appear in computation of the Jacobian determinant of the PT at $(0,0)$ in the next section.
Note that $\p_j A(h)=0$ since $A(h)$ is independent of $b$.

For computations in this section, the following three identities are useful:
\beq\label{funk1}
\int_{S^2} \frac 1{|x-y|} dS(y) = 4\pi, \quad  \int_{S^2} \frac {y_k}{|x-y|} dS(y) =  \frac {4\pi}3 x_k,
\eeq
and
\beq\label{funk2}
\int_{S^2} \frac {y_iy_k}{|x-y|} dS(y) =  \frac {16\pi}{15}\delta_{ik}+\frac {4\pi}5 x_ix_k,
\eeq
for $x \in S^2$ and for $ i, k = 1, 2, 3$, where $\Gd_{ik}$ is the Kronecker delta.

These identities can be proved using the Funk-Hecke Formula \cite[Theorem 2.22]{AtHa12Book}: for $f\in L^1(-1,1)$, $x\in S^2$ and for every homogeneous harmonic polynomial  $Y$ of degree $n$, the following formula holds
\beq\label{fk}
\int_{S^2} f(x\cdot y) Y(y) dS(y) = \Gl_n Y(x),
\eeq
where the constants $\Gl_n$ are given by
\beq\label{lambda}
\Gl_n = 2\pi \int_{-1}^{1} P_{n,3}(t) f(t) dt,
\eeq
where $P_{n,3}(t)$ are the Legendre polynomial of degree $n$ in three dimensions. Note that the constant $\Gl_n$ depends only on degree $n$. In fact, since
\beq\label{sphereid}
1-x\cdot y = \frac 12|x-y|^2,\quad x, y \in S^2,
\eeq
the relevant function $f$ for \eqnref{funk1} and \eqnref{funk2} is $f(t)= (2(1-t))^{-1/2}$. Since $P_{0,3}=1$, $P_{1,3}=t$ and $P_{2,3}=1/2(3t^2-1)$, we may apply \eqnref{fk} and \eqnref{lambda} to derive \eqnref{funk1} and \eqnref{funk2}. An additional remark may be required for the case when $i=k$ in \eqnref{funk2}. Even though $y_i^2$ is not harmonic, $y_i^2-1/3 |y|^2$ is. So we may apply \eqnref{fk} to this function and derive \eqnref{funk2} when $i=k$.

Let $U$ be the unit ball so that $\p U=S^2$. Since both sides of the equality in \eqref{funk1} are harmonic in $x\in U$, \eqref{funk1} holds for every $x \in \overline{U}$.  By the same reason,  if $i \not=k$,  \eqref{funk2} holds for every $x \in \overline{U}$.  Thus we have from \eqref{funk1} and \eqref{funk2} that for every $i, k =1, 2, 3,$ and for every  $x \in \overline{U}$
\beq\label{pre key eq 1}
\int_{S^2} \frac {x_k-y_k}{|x-y|} dS(y) = \frac{8\pi}3x_k,
\eeq
and
\beq\label{pre key eq 2}
\int_{S^2} \frac {(x_i-y_i)(x_k-y_k)}{|x-y|} dS(y) = \frac{32\pi}{15}x_ix_k \quad \mbox{if } i \not=k.
\eeq
By differentiating \eqref{pre key eq 1} in $x$, with the aid of the first equality in \eqref{funk1},  we have for every $i, k =1, 2, 3,$ and for every  $x \in \overline{U}$
\begin{equation}
\label{1st tool for the 2nd term including Gik}
\int_{S^2} \frac {(x_i-y_i)(x_k-y_k)}{|x-y|^3} dS(y) = \frac{4\pi}3\Gd_{ik}.
\end{equation}
Similarly, by differentiating \eqref{pre key eq 2} in $x$, with the aid of \eqref{pre key eq 1},  we have for every $i, k, l =1, 2, 3,$ and for every  $x \in \overline{U}$
\begin{equation}
\label{2nd tool for the 2nd term including Gik}
\int_{S^2} \frac {(x_i-y_i)(x_k-y_k)(x_l - y_l )}{|x-y|^3} dS(y) = \frac{8\pi}{15}(\Gd_{ik}x_l+\Gd_{kl}x_i+\Gd_{li}x_k),
\end{equation}
unless $i=k=l$. Even if $i=k=l$, we can recover \eqref{2nd tool for the 2nd term including Gik} by using \eqref{pre key eq 1} and $|x-y|^2 = \sum_{i=1}^3(x_i-y_i)^2$.

We now compute the first quantity in \eqnref{quantities}.
The following identities can be derived immediately from \eqnref{Rone} and \eqnref{pR1}:
\begin{align*}
R_1(0;x,y) &= 0, \\
\p_j R_1(0;x,y) &= \frac{2Y_j(x) + Y_j(y) }{r_e} + \frac{2(Y_j(x) - Y_j(y) + y\cdot\nabla_T Y_j(x))}{r_e |x-y|^2},
\end{align*}
and the following from \eqref{Rtwo} (and by taking derivatives and using \eqnref{sphereid}):
\begin{align*}
R_2(0;x,y) &=0, \\
\p_j R_2(0;x,y) & = \frac{  Y_j(x)+Y_j(y)}{r_e}.
\end{align*}
Here and afterwards, we denote $\p_j B_0(x,y):= \p_j B_b(x,y)|_{b=0}$ for $j=1,\dots,6$, just for simplicity.
We obtain from \eqnref{p1b} and the above identities that
\begin{align}\label{pjB}
\p_j B_0(x,y) &= \frac{1}{8\pi} \frac{1}{|x-y|} \left(\p_j R_1(0;x,y)  - \frac{3}{2}\p_j R_2(0;x,y) \right) \nonumber \\
&= \frac{1}{16\pi r_e} \frac{Y_j(x)- Y_j(y) }{|x-y|} + \frac{1}{4\pi r_e} \frac{Y_j(x)- Y_j(y)+ y\cdot\nabla_T Y_j(x)}{|x-y|^3} .
\end{align}

Since $Y_1$ is constant, we see from \eqref{pjB} that
$$
\p_1 B_0(x,y) = 0.
$$
For $j=2,\dots,6$, it is convenient to abuse notation and denote by $Y_j(x)$ the homogenous harmonic polynomial of order $2$ such that it is the spherical harmonic $Y_j(x)$ when $|x|=1$. If we use such notation, then by Taylor expansion we have
\beq\label{tay}
Y_j(y) = Y_j(x) + \nabla Y_j(x) \cdot (y-x) + \frac{1}{2} \sum_{i,k=1}^{3} G^j_{ik} (y_i-x_i)(y_k-x_k),
\eeq
where $G^j_{ik}= \frac{\p^2}{\p x_i\p x_k} Y_j(x)$, which are constants. Moreover, we have for $x \in S^2$
\beq\label{Tangential Gradient of harmonic polynomials}
\nabla_T Y_{j}(x)=\nabla Y_{j}(x) - (x\cdot \nabla Y_{j}(x)) x.
\eeq
Using these two identities, we obtain
\begin{align*}
Y_j(x)- Y_j(y)+y\cdot\nabla_T Y_j(x) &= Y_j(x)- Y_j(y)+y\cdot\left[\nabla Y_j(x) -(x\cdot\nabla Y_j(x)) x\right] \nonumber\\
&= (x \cdot\nabla Y_j(x))(1-x\cdot y) - \frac{1}{2} \sum_{i,k=1}^{3} G^j_{ik} (y_i-x_i)(y_k-x_k) .
\end{align*}
Recall the identity
\beq\label{Euler}
x\cdot\nabla Y_j(x) = 2Y_j(x) , \quad x \in S^2,
\eeq
which is a special case (of order 2) of Euler's theorem on homogeneous functions. Using this identity and \eqnref{sphereid}, we have
\beq\label{y2}
Y_j(x)- Y_j(y)+y\cdot\nabla_T Y_j(x) = Y_j(x)  |x-y|^2 - \frac{1}{2} \sum_{i,k=1}^{3} G^j_{ik} (y_i-x_i)(y_k-x_k).
\eeq
Plugging \eqref{y2} into \eqref{pjB}, we have
\beq\label{pbj}
\p_j B_0(x,y) = \frac{1}{8\pi r_e} \left(\frac{\frac{5}{2}Y_j(x) - \frac{1}{2}   Y_j(y)  }{|x-y|} -  \sum_{i,k=1}^{3} G^j_{ik}\frac{ (y_i-x_i)(y_k-x_k)}{|x-y|^3}\right) .
\eeq
Then,
\begin{align}
&\left\la  x_{l'} , \p_j B(0)[x_l] \right\ra  \nonumber
\\
&= \frac{1}{r_e}\int_{S^2} x_{l'}  \int_{S^2}  \left(\frac{\frac{5}{2} Y_j(x) - \frac{1}{2}  Y_j(y)  }{8\pi |x-y|} -  \sum_{i,k=1}^{3} G^j_{ik}\frac{(y_i-x_i)(y_k-x_k)}{8\pi |x-y|^3}\right) y_l dS(y) dS(x).\label{bj}
\end{align}

We now compute the right-hand side of \eqref{bj}.
The second equality in \eqref{funk1} yields
\begin{equation}
\label{first term calculation}
\int_{S^2} x_{l'}  \int_{S^2} \frac{\frac{5}{2} Y_j(x) - \frac{1}{2}  Y_j(y)  }{8\pi |x-y|} y_l dS(y) dS(x) = \frac 13\int_{S^2} x_{l'} x_lY_j(x) dS(x).
\end{equation}
Moreover, with the aid of \eqref{1st tool for the 2nd term including Gik} and \eqref{2nd tool for the 2nd term including Gik}, we compute
\begin{align}
& \sum_{i,k=1}^{3} G^j_{ik}\int_{S^2} x_{l'}  \int_{S^2} \frac{(y_i-x_i)(y_k-x_k)}{ 8\pi |x-y|^3}y_l dS(y) dS(x)\nonumber
\\
&= \sum_{i,k=1}^{3} G^j_{ik}\int_{S^2} x_{l'}  \int_{S^2}\left\{ \frac{(y_i-x_i)(y_k-x_k)(y_l-x_l)}{ 8\pi |x-y|^3} + \frac {(y_i-x_i)(y_k-x_k)x_l}{ 8\pi |x-y|^3}\right\}dS(y) dS(x) \nonumber
\\
&=\frac 2{15}\sum_{i=1}^{3} G^j_{il}\int_{S^2} x_{l'} x_i dS(x) +\frac 16\sum_{i,k=1}^{3} G^j_{ik} \Gd_{ik}\int_{S^2} x_{l'} x_l dS(x) \nonumber
\\
&=\frac 2{15}\sum_{i=1}^{3} G^j_{il}\frac{4\pi}3\Gd_{l'i}+\frac 16\sum_{i=1}^{3} G^j_{ii}\frac{4\pi}3\Gd_{l'l}  \nonumber
\\
&= \frac {8\pi}{45}G^j_{l'l}, \label{second term vanishes}
\end{align}
where the last inequality holds because $\sum_{i=1}^{3} G^j_{ii} = \GD Y_j =0$.  Then \eqref{bj} together with \eqref{first term calculation} and \eqref{second term vanishes} yields
\beq\label{general key identity 1}
\left\la  x_{l'} , \p_j B(0)[x_l] \right\ra = \frac{1}{3r_e} \int_{S^2} x_{l'} x_lY_j(x) dS - \frac{8\pi}{45r_e}G^j_{l'l}.
\eeq
Since the first term on the right-hand side above appears repeatedly, we write down the values here. Let
\beq\label{Cllj1}
C_{ll'}^j:= \int_{S^2} x_{l'} x_lY_j(x) dS.
\eeq
Then, $C_{ll'}^j$ is symmetric in $l$ and $l'$, namely, $C_{ll'}^j=C_{l'l}^j$, and for $1\leq l\leq l'\leq 3$
\beq\label{Cllj2}
C_{ll'}^j = \frac{4\pi}{15} \times
\begin{cases}
	\frac{\sqrt{15}}{3} & \mbox{if } (l,l',j)=(1,1,1), (2,2,1), (3,3,1), \\
	-\frac{1}{\sqrt{3}}  & \mbox{if } (1,1,4), (2,2,4), \\
	1 & \mbox{if } (1,1,6), (1,2,2), (1,3,5), (2,3,3), \\
	-1 & \mbox{if } (2,2,6),  \\
	\frac{2}{\sqrt{3}}  & \mbox{if } (3,3,4), \\
	0 \quad & \mbox{otherwise},
\end{cases}
\eeq
which can be seen from \eqref{Yl} and \eqref{normalize}. We then see that $\la  x_{l'} , \p_j B(0)[x_l] \ra$ is symmetric in $l$ and $l'$,
and obtain for $1\leq l\leq l'\leq 3$
\beq\label{pB list}
\left\la  x_{l'} , \p_j B(0)[x_l] \right\ra = \frac{4\pi}{45r_e} \times
\begin{cases}
	\frac{1}{\sqrt{3}}& \mbox{if } (l,l',j)=(1,1,4), (2,2,4),  \\
	-\frac{2}{\sqrt{3}} & \mbox{if } (3,3,4),  \\
	 -1 & \mbox{if } (1,1,6), (1,2,2), (2,3,3),  (1,3,5),\\
	 1 & \mbox{if } (2,2,6), \\
	0 \quad & \mbox{otherwise}.
\end{cases}
\eeq

To compute $\p_j C_{h,b}(x,y)$ at point $(h,b) = (0,0)$, we first observe that $\Ga$ and $\Gb$ given by \eqref{alpha} and \eqref{beta} take the form
$$
\Ga(0,b;x,y) = \frac{1}{2} r_i|r_ix - r_ey |^2 (1+ R_3),
$$
where
$$
R_3=R_3(b;x,y)=\frac{r_i^2 -r_e^2 - 2 r_i b(y) (x\cdot y) }{|r_ix - r_ey |^2},
$$
and
$$
\Gb(0,b;x,y) = |r_ix - r_ey |^{3}(1+ R_4)^{3/2},
$$
where
$$
R_4=R_4(b;x,y)=\frac{ b(y)(b(y) + 2r_e- 2r_i x \cdot  y) }{ |r_ix - r_ey |^{2}}.
$$
It then follows from \eqnref{GaoverGb} that
\beq\label{c0}
C_{0,b}(x,y)= \frac{r_i}{8\pi}  \frac{1}{|r_ix-r_e y|}   \frac{1+ R_3 }{(1 + R_4)^{3/2}}  ,
\eeq
and hence
\begin{align*}
& \p_j C_{0,b}(x,y) =\frac{r_i}{8\pi}  \frac{1}{|r_ix-r_e y|}  \left[\frac{\p_j R_3 }{(1+R_4)^{3/2}} - \frac{3}{2} (1+R_4)^{1/2} \frac{ (1+R_3 )  \p_j R_4}{[1+R_4]^3}\right].
\end{align*}
Since $R_4(0;x,y) =0$ if $b=0$, we have
\beq\label{dc}
\p_j C_{0,0}(x,y) =\frac{r_i}{8\pi}  \frac{1}{|r_ix-r_e y|} \left[ \p_j R_3(0;x,y)  - \frac{3}{2}   (1+R_3(0;x,y) )  \p_j R_4(0;x,y) \right].
\eeq
Straightforward computations yield the following:
$$
R_3(0;x,y) =\frac{r_i^2 -r_e^2 }{|r_ix - r_ey |^{2}}, \quad R_4(0;x,y) =0,
$$
and for $j=1,\dots, 6$
$$
\p_j R_3(0;x,y) = \frac{- 2 r_i  x\cdot y}{|r_ix - r_ey |^{2}} Y_j(y), \quad
\p_j R_4(0;x,y) = \frac{2r_e -2r_i x\cdot y}{|r_ix - r_ey |^{2}}Y_j(y) .
$$
Plugging these terms into \eqref{dc} we have
\begin{align}\label{dc0}
\p_j C_{0,0}(x,y) &=  \frac{r_i}{8\pi} \left[ \frac{-3r_e + r_i  x\cdot y}{|r_ix - r_ey |^3} - \frac{3(r_i^2-r_e^2)(r_e - r_i  x\cdot y)}{|r_ix - r_ey |^5} \right] Y_j(y) ,\quad j=1,\dots, 6.
\end{align}
Since $Y_1$ is constant, this can be written as
\begin{align}\label{dc1}
\p_j C_{0,0}(x,y) &= \frac{\p_1 C_{0,0}(x,y)}{Y_1}  Y_j(y) ,\quad j=1,\dots, 6.
\end{align}

To compute $\p_j D_{h,b}(x,y)$ at $(h,b) = (0,0)$, set
$$
\xi(h,b;x,y):= \big\la x_{e,b}(x) - x_{i,h}(y),  J_{e,b}(x) \nu_{e,b}(x)  \big\ra,
$$
and
$$
\zeta(h,b;x,y):= |x_{e,b}(x) - x_{i,h}(y)|^3.
$$
Then we have
\begin{equation*}
\begin{split}
\xi(0,b;x,y) &= \frac{1}{2}r_e |r_ix - r_ey |^2 (1+ R_5),
\end{split}
\end{equation*}
where
\begin{align*}
R_5=R_5(b;x,y)
&=\frac{r_e(r_e^2 - r_i^2) +2 b ( r_e^2 + (2r_e+ b)(r_e+b - r_i  x\cdot y)) + 2r_i(r_e + b)y\cdot\nabla_T b}{r_e |r_ix -r_ey |^2}  .
\end{align*}
We also have
$$
\zeta(0,b;x,y) = |r_ix - r_ey |^3 (1+R_6)^{3/2},
$$
where
$$
R_6 =R_6(b;x, y)=\frac{b (b + 2r_e -2r_i x\cdot y)}{|r_ix - r_ey |^2}.
$$
Then
\beq\label{d0}
D_{0,b}(x,y)= \frac{\xi(0,b;x,y)}{4\pi \zeta(0,b;x,y)} = \frac{r_e}{8\pi} \frac{1}{|r_ix-r_ey|}  \frac{1+R_5}{(1 + R_6)^{3/2}} ,
\eeq
and hence
\begin{align*}
&\p_j D_{0,b}(x,y) = \frac{r_e}{8 \pi} \frac{1}{|r_i x- r_ey|} \left[\frac{\p_j R_5 }{(1+R_6)^{3/2}} - \frac{3}{2} (1+R_6)^{1/2} \frac{ (1+R_5 )  \p_j R_6}{(1+R_6)^3}\right].
\end{align*}
Since $R_6(0;x,y) =0$ if $b=0$, we have
\beq\label{dd}
\p_j D_{0,0}(x,y) = \frac{r_e}{8 \pi} \frac{1}{|r_i x- r_ey|} \left[ \p_j R_5(0;x,y)  - \frac{3}{2}   (1+R_5(0;x,y) )  \p_j R_6(0;x,y) \right].
\eeq
Straightforward computations yield the following:
$$
R_5(0;x,y) =\frac{r_e^2 -r_i^2 }{|r_ix - r_ey |^{2}}, \quad
R_6(0;x,y) =0,
$$
and
\begin{align*}
\p_j R_5(0;x,y) &= \frac{2 (3r_e-2r_i  x\cdot y)Y_j(x) + 2r_i  y \cdot \nabla_T Y_j(x) }{|r_ix - r_ey |^{2}} , \\
\p_j R_6(0;x,y) &= \frac{2r_e -2r_i x\cdot y}{|r_ix - r_ey |^{2}}Y_j(x) .
\end{align*}
Plugging these terms into \eqref{dd} we have
$$
\p_j D_{0,0}(x,y) =  \frac{r_e}{8\pi} \frac{3r_e - r_i  x\cdot y}{|r_ix - r_ey |^3} Y_j(x) + \frac{r_e}{8\pi} \frac{2r_i y\cdot  \nabla_T Y_j(x) }{|r_ix - r_ey |^3}  -  \frac{r_e}{8\pi} \frac{3(r_e^2-r_i^2)(r_e - r_i  x\cdot y)}{|r_ix - r_ey |^5} Y_j(x).
$$
Thanks to \eqref{dc0}, this formula can be rephrased as
\beq\label{CD}
\p_j D_{0,0}(x,y) = - \frac{\p_1 C_{0,0}(x,y)}{\Gr Y_1} Y_j(x) + E_j(x,y), \quad j=1,\dots,6,
\eeq
where
\beq\label{Ej}
E_j(x,y) =  \frac{r_e}{4\pi} \frac{r_i y \cdot  \nabla_T Y_j(x) }{|r_ix - r_ey |^3} .
\eeq

We now compute $\la x_{l'}, \p_j C(0,0)[x_l] + \Gr \p_j D (0,0)[x_l] \ra$. Thanks to \eqref{dc1} and \eqref{CD}, we have
\begin{align*}
& \la x_{l'}, \p_j C(0,0)[x_l] + \Gr \p_j D (0,0)[x_l] \ra \\
& = \int_{S^2} x_{l'} \int_{S^2} \left( \frac{\p_1 C_{0,0}(x,y)}{Y_1} (Y_j(y) - Y_j(x)) + \Gr E_j(x,y) \right)y_l dS(y) dS(x).
\end{align*}
Since $\p_1 C_{0,0}(x,y) $ is a function of variable $x\cdot y$, we may apply Funk-Hecke formula \eqref{fk} to see that
\begin{align*}
&\int_{S^2} x_{l'} \int_{S^2} \frac{\p_1 C_{0,0}(x,y)}{Y_1} (Y_j(y) - Y_j(x)) y_l dS(y) dS(x) \\
& = \frac{1}{Y_1} \int_{S^2} \int_{S^2}  x_{l'} \p_1 C_{0,0}(x,y) Y_j(y)  y_l dS(y) dS(x) - \frac{1}{Y_1}  \int_{S^2} \int_{S^2} y_l   \p_1 C_{0,0}(x,y)  Y_j(x)  x_{l'} dS(x) dS(y)\\
& = 0.
\end{align*}
Thus we have
\beq\label{pE}
\la x_{l'}, \p_j C(0,0)[x_l] + \Gr \p_j D (0,0)[x_l] \ra =  \Gr \left\la  x_{l'} , E_j[x_l] \right\ra,
\eeq
where
\beq
E_j[ x_l](x) := \int_{S^2} E_j(x,y)   y_l \, dS(y).
\eeq

We now compute the term $\left\la  x_{l'} , E_j[x_l] \right\ra$. Clearly $E_1[ x_l] = 0$. For $j=2,\dots,6$ and for $t=1,2,3$, by \eqref{tay} we have
$$
\frac {\p Y_j}{\p y_t}(y) = \frac {\p Y_j}{\p y_t}(x) + \sum_{k=1}^3G^j_{tk}(y_k-x_k).
$$
Then  it follows from Euler's theorem \eqnref{Euler} that
\beq\label{Euler 2}
y\cdot\nabla Y_j(x) = 2Y_j(y) - \sum_{t,k=1}^3G_{tk}^j(y_k-x_k)y_t.
\eeq
Hence, by \eqref{Tangential Gradient of harmonic polynomials} and \eqref{Ej}, we have for $j=2,\dots,6$
\beq\label{useful representation of Ej}
E_j(x,y) = \frac {r_er_i\left[ 2Y_j(y)-2Y_j(x)x\cdot y- \sum_{t,k=1}^3G_{tk}^j(y_k-x_k)y_t \right]}{4\pi |r_ix-r_ey|^3}.
\eeq
Therefore we have
\begin{align}
&\left\la  x_{l'} , E_j[x_l] \right\ra =\int_{S^2} x_{l'} \int_{S^2} E_j(x,y)  y_l \, dS(y) dS(x) =\nonumber
\\
& \frac {r_er_i }{4\pi}\!\!\int_{S^2} \int_{S^2}\!\!\left\{\frac{2x_{l'}Y_j(y)y_l}{ |r_ix-r_ey|^3}-\frac{2x_{l'}Y_j(x)y_lx\cdot y}{ |r_ix-r_ey|^3}- \!\!\!\sum_{t,k=1}^3G_{tk}^j\frac{x_{l'}y_ky_ty_l-x_{l'}x_ky_ty_l}{ |r_ix-r_ey|^3}\right\}\!\, dS(y) dS(x). \label{pre calculation for the double integral of E_j}
\end{align}

Before calculating \eqref{pre calculation for the double integral of E_j}, we prepare several integral formulas. Since formula  \eqref{funk1} holds for every $x \in \overline{U}$, we infer that for  every $x \in \overline{U}$ and every $k=1,2,3,$
\beq\label{the 1st integral formula we start with}
 \int_{S^2} \frac 1{|r_ix-r_ey|} dS(y) = \frac{4\pi}{r_e},  \quad \int_{S^2} \frac {y_k}{|r_ix-r_ey|} dS(y) =  \frac {4\pi r_i}{3r_e^2} x_k.
\eeq
Moreover, by the Funk-Hecke formula, we have for every $x, y \in S^2$ and for $k=1,2,3$
\beq\label{the 2nd integral formula 1}
\int_{S^2} \frac 1{|r_ix-r_ey|^3} dS(y) = \frac {4\pi}{r_e(r_e^2-r_i^2)},
\eeq
and
\beq\label{the 2nd integral formula 2}
\int_{S^2} \frac {y_k}{|r_ix-r_ey|^3} dS(y) = \frac {4\pi r_i}{r_e^2(r_e^2-r_i^2)}x_k, \quad
\int_{S^2} \frac {x_k}{|r_ix-r_ey|^3} dS(x) = -\frac {4\pi r_e}{r_i^2(r_e^2-r_i^2)}y_k.
\eeq
Here we used the fact that the denominator of the integrands never vanish. The first equality in \eqref{the 2nd integral formula 2} can be also obtained by differentiating the first equality in \eqref{the 1st integral formula we start with} with respect to $x$, with the aid of  \eqref{the 2nd integral formula 1}. Moreover, by differentiating the second equality in \eqref{the 1st integral formula we start with} with respect to $x$, with the aid of the first equality of \eqref{the 2nd integral formula 2},  we have for $k, t =1, 2, 3,$ and for every  $x \in \overline{U}$
\beq\label{the 3rd integral formula}
\int_{S^2} \frac {y_ky_t}{|r_ix-r_ey|^3} dS(y) = \frac {4\pi}{3r_e^3}\Gd_{kt} +\frac {4\pi r_i^2}{r_e^3(r_e^2-r_i^2)}x_kx_t.
\eeq

Using \eqref{the 2nd integral formula 2} and \eqref{the 3rd integral formula}, we have
\begin{align*}
\int_{S^2} \int_{S^2}\frac{2x_{l'}Y_j(y)y_l}{ |r_ix-r_ey|^3}\, dS(y) dS(x) &= -\frac {8\pi r_e}{r_i^2(r_e^2-r_i^2)}\int_{S^2} y_{l'}y_lY_j(y) dS(y), \\
\int_{S^2} \int_{S^2}\frac{2x_{l'}Y_j(x)y_lx\cdot y}{ |r_ix-r_ey|^3}\, dS(y) dS(x) &= \frac {8\pi(r_e^2+2r_i^2)}{3r_e^3(r_e^2-r_i^2)}\int_{S^2} x_{l'}x_lY_j(x) dS(x),
\end{align*}
and
\begin{align*}
&\int_{S^2} \int_{S^2} \sum_{t,k=1}^3G_{tk}^j\frac{x_{l'}y_ky_ty_l-x_{l'}x_ky_ty_l}{ |r_ix-r_ey|^3}\, dS(y) dS(x) \\
&= -\frac{8\pi(r_e^4+r_i^4)}{r_e^3r_i^2(r_e^2-r_i^2)}\int_{S^2} x_{l'}x_lY_j(x) dS(x) - \frac{16\pi^2}{9r_e^3}G^j_{l'l},
\end{align*}
where we used the fact that $\sum_{t,k=1}^3 G_{tk}^j x_t x_k= 2Y_j(x)$. Plugging these identities into \eqref{pre calculation for the double integral of E_j} yields
\beq\label{general key identity 2 for Ej}
\left\la  x_{l'} , E_j[x_l] \right\ra = -\frac {2r_i}{3r_e^2} \int_{S^2} x_{l'}x_lY_j(x) dS(x) + \frac {4\pi r_i}{9r_e^2}G^j_{l'l}.
\eeq
Thus we see from \eqref{pE} that $\la x_{l'}, \p_j C(0,0)[x_l] + \Gr \p_j D (0,0)[x_l] \ra$ is symmetric in $l,l'$ and for $1\leq l\leq l'\leq 3$:
\begin{align}
&\left\la x_{l'}, \p_j C(0,0)[x_l] + \Gr \p_j D (0,0)[x_l] \right\ra = \Gr \left\la  x_{l'} , E_j[x_l] \right\ra \nonumber\\
&= \frac{4\pi \Gr^2}{15r_e} \times
\begin{cases}
	-\frac{1}{\sqrt{3}}& \mbox{if } (l,l',j)=(1,1,4), (2,2,4),  \\
	\frac{2}{\sqrt{3}} & \mbox{if } (3,3,4),  \\
	1 & \mbox{if } (1,1,6), (1,2,2), (2,3,3), (1,3,5),\\
	- 1 & \mbox{if } (2,2,6), \\
	0 \quad & \mbox{otherwise}.
\end{cases}\label{CD list final}
\end{align}

\section{Proof of Theorem \ref{main_thm}}\label{sec:proof}

In this section we prove Theorem \ref{main_thm} by showing that $m_{ll'}$ satisfies the hypothesis of the implicit function theorem: continuity in $(h,b)$, continuous differentiability in $b$, and \eqnref{bJacob}. Here we recall the implicit function theorem in the following form \cite{KP}:
\begin{thm}\label{thm:ift}
Let $X$ be a Banach space. Let $U\times V$ be an open subset of $X\times \Rbb^6$. Suppose that
$$
F=(F_1,\dots,F_6): (x,y)\in U\times V\mapsto  \Rbb^6
$$
is continuous and has the property that the derivative of $F$ with respect to $y$ exists and is continuous at each point of $U\times V$. Further assume that at point $(x_0,y_0)\in U\times V$,
$$
F(x_0,y_0)=0 \quad \mbox{and} \quad   \frac{\p F}{\p y}(x_0,y_0) \neq 0.
$$
Then there exist neighborhood $N_1\subset U$ of $x_0$ and neighborhood $N_2\subset V$ of $y_0$ such that, for each $x$ in $N_1$, there is a unique $y\in N_2$ satisfying
$$
F(x,y)=0.
$$
The function $\hat y$, thereby uniquely defined near $x_0$ by the condition $\hat y(x)=y$, is continuous.
\end{thm}

Let $m_{ll'}(h,b):=m_{ll'}(D_h, \GO_b)$ as before. By definition \eqref{PT0}, $m_{ll'}(h,b)$, $l, l'=1,2,3$, are given by
$$
m_{ll'}(h,b)=\int_{\p D_h}x_{l'}\Gvf_1^{(l)} \, dS+\int_{\p \GO_b}x_{l'} \Gvf_2^{(l)} \, dS,
$$
where $\Gvf^{(l)}= (\Gvf_1^{(l)},\Gvf_2^{(l)})\in L^2_0(\p D_h)\times  L^2_0(\p \GO_b)$ is the unique solution to \eqnref{pp}.
Using changes of variables \eqnref{icoordi} and \eqnref{ecoordi}, we see that
\beq\label{m11pre}
m_{ll'}(h,b) =\int_{S^2} (r_i+ h(x))x_{l'}\, f_{h,b,1}^{(l)}(x) \, dS + \int_{S^2} (r_e+ b(x))x_{l'} \, f_{h,b,2}^{(l)}(x) \, dS,
\eeq
where
$$
f_{h,b,1}^{(l)}(x):= \Gvf^{(l)}_1(x_{i,h}(x)) J_{i,h}(x), \quad f_{h,b,2}^{(l)}(x):= \Gvf_2^{(l)}(x_{e,b}(x)) J_{e,b}(x).
$$
Let $f^{(l)}_{h,b}= (f^{(l)}_{h,b,1}, f^{(l)}_{h,b,2})^\top$ and
\beq\label{p}
p(h,b):= (r_i+ h(x), r_e+ b(x))^\top,
\eeq
where $\top$ denotes the transpose. Then, we have
\beq\label{m11def}
m_{ll'}(h,b) = \left\la x_{l'} p(h,b), f_{h,b}^{(l)} \right\ra.
\eeq

Note that $f^{(l)}_{h,b}$ is the solution of
\beq\label{inteqn}
\Acal (h,b) [f^{(l)}_{h,b}] = g^{(l)}_{h,b}:= -
\begin{bmatrix} \nu^{(l)}_{\p D_h}(x_{i,h}(x)) J_{i,h}(x) \\ \nu_{\p\GO_b}^{(l)}(x_{e,b}(x)) J_{e,b}(x) \end{bmatrix}.
\eeq
We see from \eqnref{inormal} and \eqnref{enormal} that $g^{(l)}_{h,b}$, $l=1,2,3$, is given by
\beq\label{gl}
g^{(l)}_{h,b} = - \begin{bmatrix} (r_i + h(x))[(r_i+h(x))x_l - \nabla_T h(x)_l] \\
(r_e + b(x))[(r_e+b(x))x_l - \nabla_T b(x)_l] \end{bmatrix}.
\eeq

In what follows, we show that the mapping $F:= (m_{11}, m_{12}, m_{13},m_{22}, m_{23}, m_{33})$ satisfies hypothesis of Theorem \ref{thm:ift}.

\medskip
\noindent{\bf Continuity in $(h,b)$}. We only prove continuity of $m_{11}$ since the others can be handled in the same way.

Suppose $k \in W^{2,\infty}(S^2)$ and $d \in W_6$. Then we have
$$
\Acal (k,d) [f^{(1)}_{k,d}] = g^{(1)}_{k,d}.
$$
Thus,
$$
\Acal (k,d) [f^{(1)}_{k,d} - f^{(1)}_{h,b}] = - (\Acal (k,d) - \Acal (h,b))[f^{(1)}_{h,b}] + ( g^{(1)}_{k,d} - g^{(1)}_{h,b}) .
$$
We then infer using Corollary \ref{cor} that
$$
\left\| f^{(1)}_{k,d} - f^{(1)}_{h,b} \right\|_2 \le K \left( \left\| (\Acal (k,d) - \Acal (h,b))[f^{(1)}_{h,b}] \right\|_2 + \left\| g^{(1)}_{k,d} - g^{(1)}_{h,b} \right\|_2 \right)
$$
for some constant $K$ independent of $(k,d)$ as long as $\| k \|_{2,\infty}$ and $| d|_\infty$ are sufficiently small. We then infer from \eqnref{contitwo} that
$$
\left\| (\Acal (k,d) - \Acal (h,b))[f^{(1)}_{h,b}] \right\|_2 \to 0
$$
as $\| k-h \|_{2,\infty} + |d-b|_\infty \to 0$. It is obvious from \eqnref{gl} that $\| g^{(1)}_{k,d} - g^{(1)}_{h,b} \|_2 \to 0$. Thus we have $\| f^{(1)}_{k,d} - f^{(1)}_{h,b} \|_2 \to 0$. We then conclude using \eqnref{m11def} that $m_{11}(k,d)-m_{11}(h,b) \to 0$ as $\| k-h \|_{2,\infty} + |d-b|_\infty \to 0$.

\medskip
\noindent{\bf Continuous differentiability in $b$}. By differentiating \eqnref{inteqn} with respect to the $b_j$-variable, we have
$$
\Acal (h,b) [\p_j f^{(1)}_{h,b}] = \p_j g^{(1)}_{h,b}- \p_j \Acal (h,b) [f^{(1)}_{h,b}],
$$
namely,
\beq\label{derieqn}
\p_j f^{(1)}_{h,b} = \Acal (h,b)^{-1} \left[ \p_j g^{(1)}_{h,b} - \p_j \Acal (h,b) [f^{(1)}_{h,b}] \right].
\eeq
We mention that this argument is formal since we take the derivative of $f^{(1)}_{h,b}$ without proving its existence. However, this formal argument can be justified easily.

It is clear from \eqnref{gl} that $\p_j g^{(1)}_{h,b}$ is continuous in $(h,b)$. Then Corollary \ref{cor}, Proposition \ref{prop:deri} and continuity of $f^{(1)}_{h,b}$ in $(h,b)$ imply that $\p_j f^{(1)}_{h,b}$ is continuous in $(h,b)$. We then obtain from \eqnref{m11def} that
$$
\p_j m_{11}(h,b) = \left\la x_1 \p_j p(h,b), f_{h,b}^{(1)} \right\ra + \left\la x_1 p(h,b), \p_j f_{h,b}^{(1)} \right\ra,
$$
which shows that $\p_j m_{11}(h,b)$ is continuous in $(h,b)$.

\medskip
\noindent{\bf Proof of \eqnref{bJacob}}. For ease of notation we put
$$
\psi_l(x):= x_l, \quad l=1,2,3.
$$
Then derivatives of $m_{ll'}$ takes the following form
\begin{align}
\p_j m_{ll'}(0,0) &= \left\la \psi_{l'} \p_j p(0,0), f_{0,0}^{(l)} \right\ra + \left\la \psi_{l'} p(0,0), \p_j f_{0,0}^{(l)} \right\ra. \label{mderi}
\end{align}

To compute terms on the right-hand side above, we first show that $\Acal (0,0)$ preserves the space spanned by $\psi_l (1,0)^\top$ and $\psi_l (0,1)^\top$, $l=1,2,3$, and $\Acal (0,0)^{-1}$ on that space is given by
\beq\label{Ainv}
\Acal(0,0)^{-1} \begin{bmatrix} a \psi_l\\ b \psi_l \end{bmatrix}
= \Gg_1  \begin{bmatrix}
	(-\mu +\frac{1}{6}) &  \frac{1}{3} \rho^2 \\
    - \frac{2}{3}\rho & \rho^3 (\mu +\frac{1}{6}) \end{bmatrix} \begin{bmatrix} a \psi_l\\ b \psi_l \end{bmatrix},
\eeq
where $\mu$ is the number defined in \eqnref{lambdamu} and
\beq
\Gg_1 = \dfrac{1}{\Gr^3 (1/2+\mu)(1/2-\mu)}.
\eeq
To do so, we need to compute $A(0)[\psi_l]$, $B(0)[\psi_l]$, $C(0,0)[\psi_l]$ and $D(0,0)[\psi_l]$.

We see from \eqnref{Ah}, \eqnref{keb_ker} and \eqnref{sphereid} that
$$
A_0(x,y) = B_0 (x,y) =\frac{1}{4\pi}\frac{\la x -y, x  \ra}{|x -y|^3} =\frac{1}{8\pi}\frac{1}{|x -y|}, \quad x,y\in S^2.
$$
Thus it follows from the second identity in \eqnref{funk1} that
\beq
A(0)[\psi_l] = B(0)[\psi_l] = \frac{1}{6} \psi_l.
\eeq
We see from \eqref{C} and \eqref{Dkernel} that
$$
C_{0,0}(x,y)  = \frac{1}{4\pi}\frac{\la r_i x-r_ey, x\ra}{|r_ix-r_e y|^3}r_i^2 \quad\mbox{and}\quad
D_{0,0}(x,y)  = \frac{1}{4\pi}\frac{\la r_e x-r_iy, x\ra}{|r_ex-r_i y|^3}r_e^2.
$$
Thus, \eqnref{the 2nd integral formula 1} and \eqnref{the 2nd integral formula 2} yield
\beq
C(0,0)[\psi_l]  = -\frac{\rho^2}{3} \psi_l \quad\mbox{and}\quad
D(0,0)[\psi_l] =  \frac{2\rho}{3} \psi_l.
\eeq
Thus,
\begin{equation*}
\Acal(0,0)\begin{bmatrix}
a \psi_l\\
b \psi_l
\end{bmatrix} =\begin{bmatrix}
-\Gl I +A(0)  &  C(0,0)  \\
D(0,0)&   -\Gm I+B(0)
\end{bmatrix}\begin{bmatrix}
a \psi_l\\
b \psi_l
\end{bmatrix}=
\begin{bmatrix}
-\Gl + 1/6  &  -\rho^2/3  \\
2\rho/3 &   -\Gm+1/6
\end{bmatrix}\begin{bmatrix}
a \psi_l\\
b \psi_l
\end{bmatrix}.
\end{equation*}
The desired formula \eqnref{Ainv} now follows thanks to the relation $\Gl=\frac{1}{6}-\rho^3(\mu+\frac{1}{6})$, which comes from \eqref{neutral}(the neutrality condition) and \eqref{lambdamu} (definitions of $\Gl$ and $\Gm$).

We now compute the first term on the right-hand side of \eqref{mderi}. Since $g^{(l)}_{0,0} = - \psi_l (r_i^2 , r_e^2)^\top$,
we have
\beq\label{f0}
f^{(l)}_{0,0} =  \Acal(0,0)^{-1}[g^{(l)}_{0,0}] = \psi_l V_1,
\eeq
where the constant vector $V_1$ is defined by
\beq\label{Vone}
V_1 := \Gg_2 r_e^2 \begin{bmatrix} - 1  \\ \Gr \end{bmatrix} \quad\mbox{with}\quad
\Gg_2 = \dfrac{1}{\Gr (1/2+\mu)} = \Gr^2 (1/2-u) \Gg_1.
\eeq
By \eqref{p}, $\p_j p(0,0)= (0, Y_j)^\top$, and hence $\p_j p(0,0) \cdot V_1= \Gg_2 r_e^2 \Gr Y_j$. Therefore,
\beq\label{1term}
\left\la \psi_{l'} \p_j p(0,0), f_{0,0}^{(l)} \right\ra = \Gg_2 r_e^2 \Gr C_{ll'}^j,
\eeq
where $C_{ll'}^j$ is defined and computed in \eqnref{Cllj1} and \eqnref{Cllj2}.

To compute the second term on the right-hand side of \eqref{mderi}, namely, $\la \psi_{l'} p(0,0), \p_j f_{0,0}^{(l)} \ra$, we first observe from \eqnref{derieqn} that
$$
\p_j f^{(l)}_{0,0} = \Acal (0,0)^{-1} \left[ \p_j g^{(l)}_{0,0} - \p_j \Acal (0,0) [f^{(l)}_{0,0}] \right].
$$
Thus we have
$$
\left\la \psi_{l'} p(0,0), \p_j f_{0,0}^{(l)} \right\ra = \left\la  (\Acal (0,0)^{-1})^*[\psi_{l'} p(0,0)], \p_j g^{(l)}_{0,0} - \p_j \Acal (0,0) [f^{(l)}_{0,0}] \right\ra,
$$
where $(\Acal (0,0)^{-1})^*$ is the adjoint operator of $\Acal (0,0)^{-1}$. In view of \eqnref{Ainv}, we have
\beq\label{V2}
(\Acal (0,0)^{-1})^*[\psi_{l'} p(0,0)] = \psi_{l'} \Gg_1 r_e \Gr (1/2+\mu) \begin{bmatrix} -1\\ \Gr^2  \end{bmatrix} =: \psi_{l'} V_2,
\eeq
and hence
\beq\label{2term}
\left\la \psi_{l'} p(0,0), \p_j f_{0,0}^{(l)} \right\ra = \left\la \psi_{l'} V_2, \p_j g^{(l)}_{0,0} - \p_j \Acal (0,0) [f^{(l)}_{0,0}] \right\ra .
\eeq

For ease of notation, let $V_3:= (0,1)^\top$.
Then, one can see from \eqnref{gl} that $\p_j g^{(l)}_{0,0}$ are given by the following:
\beq
\p_j g^{(l)}_{0,0} =
\begin{cases}
- 2 r_e \psi_l Y_j V_3, & j=1, \\
- \left(4r_e \psi_l Y_j - r_e \frac{\p Y_j}{\p x_l} \right) V_3, & j \neq 1.
\end{cases}
\eeq
Note $V_2 \cdot V_3= \Gg_1 r_e \Gr^3 (1/2+\mu)$. Now straightforward but tedious computations yield for $1\leq l\leq l'\leq 3$
\beq\label{g}
\left\la \psi_{l'} V_2,  \p_j g^{(l)}_{0,0}  \right\ra = \frac{4\pi }{15} \Gg_1 r_e^2 \Gr^3 (1/2+\mu) \times
\begin{cases}
	 - \frac{2\sqrt{15}}{3} & \mbox{if } (l,l',j)= (1,1,1), (2,2,1), (3,3,1), \\
	- \frac{1}{\sqrt{3}} & \mbox{if }  (1,1,4),(2,2,4),\\
	1 & \mbox{if } (1,1,6), (1,2,2), (1,3,5), (2,3,3),\\
	- 1 & \mbox{if } (2,2,6),\\
	 \frac{2}{\sqrt{3}} & \mbox{if }  (3,3,4),\\
	0 \quad & \mbox{otherwise}.
\end{cases}
\eeq

It now remains to calculate $\la \psi_{l'} V_2, \p_j \Acal (0,0) [f^{(l)}_{0,0}] \ra$. By \eqref{CD} and \eqref{f0}, we have
\begin{align*}
&\left\la \psi_{l'} V_2, \p_j \Acal (0,0) [f^{(l)}_{0,0}] \right\ra \\
&= \Gg_1 r_e \Gr (1/2+\mu) \Gg_2 r_e^2 \int_{S^2}  \psi_{l'}  \begin{bmatrix} -1\\ \Gr^2 \end{bmatrix}  \cdot \begin{bmatrix}
\Gr \p_j C(0,0)[\psi_l]\\
 - \p_j D(0,0)[\psi_l]  +\Gr  \p_j B(0)[\psi_l]
\end{bmatrix} dS \\
&= \Gg_1 r^3_e \Gr^2 (1/2+\mu) \Gg_2 \int_{S^2}  \psi_{l'} \left[ -\p_j C(0,0)[\psi_l] - \Gr \p_j D (0,0)[\psi_l] +\Gr^2  \p_j B(0)[\psi_l]\right] dS.
\end{align*}
It then follows from \eqref{pB list} and \eqref{CD list final} that for $1\leq l\leq l'\leq 3$:
\beq\label{f}
\left\la \psi_{l'} V_2, \p_j \Acal (0,0) [f^{(l)}_{0,0}] \right\ra  = \frac{16\pi}{45}\Gr^3 r_e^2 \Gg_1 \times
\begin{cases}
	\frac{1}{\sqrt{3}}  & \mbox{if } (l,l',j)=(1,1,4), (2,2,4), \\
	-1  & \mbox{if } (1,1,6),  (1,2,2), (1,3,5), (2,3,3), \\
	 1 & \mbox{if } (2,2,6),  \\
	-\frac{2}{\sqrt{3}} & \mbox{if } (3,3,4), \\
	0 \quad & \mbox{otherwise}.
\end{cases}
\eeq

We then have from \eqref{mderi}, \eqref{1term}, \eqref{2term}, \eqref{g} and \eqref{f} that for $1\leq l\leq l'\leq 3$:
\begin{equation*}
\p_j m_{ll'}(0,0) = \pi \Gr^3 r_e^2 \Gg_1 \times
\begin{cases}
	-\frac{4}{3\sqrt{15}}  (\frac{1}{2} + 3\mu)& \mbox{if } (l,l',j)=(1,1,1), (2,2,1), (3,3,1), \\
	-\frac{28}{45\sqrt{3}} & \mbox{if } (1,1,4), (2,2,4), \\
	\frac{28}{45} & \mbox{if } (1,1,6),  (1,2,2), (1,3,5), (2,3,3), \\
	-\frac{28}{45} & \mbox{if } (2,2,6),  \\
	\frac{56}{45\sqrt{3}} & \mbox{if } (3,3,4), \\
	0 \quad & \mbox{otherwise}.
\end{cases}
\end{equation*}
Thus,
\begin{align*}
&\frac{\p (m_{11}, m_{12}, m_{13},m_{22}, m_{23}, m_{33})}{\p (b_1,b_2, b_3,b_4,b_5, b_6)} (0,0) \\
&= \left(\pi \Gr^3 r_e^2 \Gg_1 \right)^6  \det
\begin{bmatrix}
-\frac{4}{3\sqrt{15}} (\frac{1}{2 } + 3\mu)& 0& 0 &-\frac{28}{45\sqrt{3}}  & 0 & \frac{28}{45 }\\
0 & \frac{28}{45 } & 0 &0&0&0 \\
0 &0&0&0& \frac{28}{45 } &0 \\
-\frac{4}{3\sqrt{15}} (\frac{1}{2 } + 3\mu)& 0& 0 &-\frac{28}{45\sqrt{3}}  & 0 & -\frac{28}{45 }\\
0 & 0 & \frac{28}{45 } &0&0&0 \\
-\frac{4}{3\sqrt{15}} (\frac{1}{2 } + 3\mu) & 0& 0 &\frac{56}{45\sqrt{3}} & 0 & 0
\end{bmatrix} \\
&=  \left(\Gr^3 r_e^2 \gamma_1 \pi  \right)^6  \frac{4}{3\sqrt{15}} (\frac{1}{2 } + 3\mu) \left(\frac{28}{45 } \right)^4 \frac{28}{45\sqrt{3}} 6 \neq 0.
\end{align*}
Thus, \eqnref{bJacob} is proved.

\begin{rem}
By switching roles of $h$ and $b$, let $M(b,h)$ be the polarization tensor associated with domain $(\Omega_h,D_b)$. Similar computations yield
\begin{equation*}
\frac{\p (m_{11}, m_{12}, m_{13},m_{22}, m_{23}, m_{33})}{\p (b_1,b_2, b_3,b_4,b_5, b_6)} (0,0) =  \left(\Gr r_e^2 \Gg_1 \pi  \right)^6   \frac{4}{3\sqrt{15}} (\frac{1}{2 } + 3\mu) \left(\frac{28}{45 } \right)^4 \frac{28}{45\sqrt{3}} 6 \neq 0.
\end{equation*}
Thus we have Theorem \ref{main_thm2}.
\end{rem}

\section*{Conclusion}

In this paper we consider the problem of the PT-vanishing inclusion (or the weakly neutral inclusion) of the core-shell structure: Given a domain of arbitrary shape find a domain enclosing the given domain so that the core-shell structure is PT-vanishing. We show that such a domain for shell exists if the given domain is a small perturbation of a ball. The result of this paper is a proof of existence. As far as we are aware of, there is no known method of constructing such domains for shells. Even shells for ellipses or ellipsoids are not known. Thus it is quite interesting to find a way to construct shells for the PT-vanishing structure. In this regard, we mention that there is a numerical attempt to construct the PT-vanishing structure using shape derivative \cite{FKL}.


\end{document}